\documentclass[12pt]{article}
\usepackage{amsmath}
\usepackage{amssymb}
\usepackage{amsfonts}
\usepackage{amsmath, dsfont}
\usepackage{amsthm}
\usepackage[margin=1in]{geometry}
\usepackage{graphicx}
\usepackage{graphics}
\usepackage{subfigure}
\usepackage{color}

\newtheorem{theorem}{Theorem}
\newtheorem{lemma}{Lemma}

\newcommand{\se}{ {\sigma}_{\mbox{\rm eff}}}
\begin{document}

\bibliographystyle{plain}

\title{Branching Random Walks in Time Inhomogeneous Environments}
%\author{Ming Fang\thanks{} \and Ofer Zeitouni\thanks{}}
\author{Ming Fang\thanks{School of Mathematics, University of Minnesota,
206 Church St. SE, Minneapolis, MN 55455, USA.
The work of this author was partially
supported by NSF grant DMS-0804133}
\and
Ofer Zeitouni\thanks{School of Mathematics, University of Minnesota,
206 Church St. SE, Minneapolis, MN 55455, USA and Faculty of Mathematics,
Weizmann
Institute, POB 26, Rehovot 76100, Israel.
The work of this author was partially
supported by NSF grant DMS-0804133, a grant from the Israel science foundation, and the Taubman professorial chair at the Weizmann
Institute.
}}

\date{November 18, 2011}
\maketitle

{\abstract We study the maximal displacement of branching random
walks in a class of time inhomogeneous environments.
Specifically, binary branching random walks with Gaussian increments
will be considered, where the variances of the increments change over time
macroscopically.
We find the asymptotics of the maximum up to an $O_P(1)$ (stochastically bounded) error, and focus on
the following phenomena: the profile of the variance matters, both to the
leading (velocity) term and to the logarithmic correction term, and the latter
exhibits a phase transition.}

\section{Introduction}

Branching random walks and their maxima
have been
studied mostly in space-time
homogeneous environments (deterministic or random). For work on the
deterministic homogeneous case of relevance to our
study we refer to \cite{Bramson78_BBM} and the recent
\cite{Addario-BerryReed09} and  \cite{Aidekon10}.
For the random environment
case, a sample of relevant papers is
%
%that is, the laws governing the randomness of the model stay the same. Also there were papers dealing with branching random walks in random environments, for example,
\cite{GantertMullerPopovVachkovskaia10, GrevenHollander92, HeilNakashimaYoshida11, HuYoshida09, Liu07, MachadoPopov00, Nakashima11}. As is well documented in these references,
under reasonable hypotheses, in the homogeneous case
the maximum grows linearly, with a logarithmic correction,
and
is tight around its median.

Branching random walks are also studied under some space inhomogeneous environments. A sample of those papers are \cite{BerestyckiBrunetHarrisHarris10, DoeringRoberts11, EnglanderHarrisKyprianou10, GitHarrisHarris07, HarrisHarris09, HarrisWilliams96, Koralov11}.

Recently, Bramson and Zeitouni \cite{BramsonZeitouni09} and Fang \cite{Fang11} showed
that the maxima of branching random walks, recentered around their median,
are still tight in time
inhomogeneous environments satisfying certain uniform regularity assumptions,
in particular,  the laws of the increments
can vary with respect to time and the walks may have some local dependence.
A natural question is to ask, in that situation,
what is the asymptotic behavior of the maxima.
Similar questions were discussed in the context of branching Brownian
motion using PDE techniques,
%have also been studied in the field of partial differential equations,
see e.g. Nolen and Ryzhik \cite{NolenRyzhik09},
using the fact that
the distributions of the maxima satisfy the KPP equation
whose solution exhibits a traveling wave phenomenon.

In all these models,
while the linear traveling speed of the maxima is a relatively easy
consequence of the large deviation principle,
the evaluation of the
second order correction term, like the ones in Bramson \cite{Bramson78_BBM}
and Addario-Berry and Reed \cite{Addario-BerryReed09}, is more involved and requires a
detailed analysis of the walks; to our knowledge, it has so far
only been performed
in the time homogeneous case.

Our goal is to start exploring the time inhomogeneous setup.
As we will detail below, the situation, even in the simplest setting, is
complex and, for example, the order in which inhomogeneity presents itself
matters, both in the leading term and in the correction term.

In this paper, in order to best describe the phenomenon discussed above,
we focus on the  simplest case of binary branching random walks
where the diffusivity of the particles
takes two distinct values as a function of time.

We now describe the setup in detail.
%We consider branching random walks
%whose increments are not always identically distributed. In particular,
For $\sigma>0$, let
$N(0,\sigma^2)$ denote the
 normal distributions with mean zero and variance $\sigma^2$.
Let $n$ be an integer, and let $\sigma_1^2,\sigma_2^2>0$ be given. \
We start the system with one particle at location 0 at time 0.
Suppose that $v$ is a particle at location $S_v$ at time $k$.
Then $v$ dies at time $k+1$ and gives birth to two particles $v1$ and $v2$,
and each of the two offspring ($\{vi,i=1,2\}$) moves independently to a
new location $S_{vi}$ with the increment $S_{vi}-S_{v}$ independent
of $S_v$ and
distributed as $N(0,\sigma_1^2)$ if $k<n/2$ and as
$N(0,\sigma_2^2)$ if $n/2\leq k<n$.
Let $\mathds{D}_n$ denote the collection
of all particles at time $n$. For a particle
$v\in \mathds{D}_n$ and $i<n$, we let
$v^i$ denote
the $i$th level ancestor of $v$, that is the unique
element of $\mathds{D}_i$ on the geodesic connecting $v$ and the root.
We study the maximal displacement $M_n=\max_{v\in\mathds{D}_n}S_v$ at
time $n$, for $n$ large. \footnote[1]{Since one can understand a branching random walk as a `competition' between branching and random walk, one may get similar results by fixing the variance and changing the branching rate with respect to time.}

It will be clear that the analysis extends to a wide class of inhomogeneities
with finitely many values and `macroscopic' change (similar to the description in the previous paragraph), and to the Galton-Watson
setup.
%We show that the order in which the different diffusivities appear matters, and the %behavior
%of the maxima is very different in the cases where large variances of the increments
%precede, or follow, smaller variances.
A
%and the branching random walks can behave quite differently. This analysis sheds some light on branching random walks with finitely many different environments or even infinitely many environments, while a
universal result that will allow for continuous change of the variances
is more complicated, is expected to present different correction
terms, and is the subject of further study.

In order to describe the results in a concise way, we recall the notation $O_P(1)$ for stochastically boundedness. That is, if a sequence of random variables $R_n$ satisfies $R_n=O_P(1)$, then, for any $\epsilon>0$, there exists an $M$ such that $P(|R_n|> M)<\epsilon$ for all $n$.

An interesting feature of $M_n$ is that the asymptotic behavior
%does not possess a closed form, rather, it
depends on the order
relation between $\sigma_1^2$ and $\sigma_2^2$. That is, while
\begin{equation}\label{uni_max}
M_n=\left(\sqrt{2\log 2}\;\se\right)n-\beta \frac{\se}{\sqrt{2\log 2}}\log n +O_P(1)
\end{equation}
is true for some choice of $\se$ and $\beta$, $\se$ and $\beta$ take different
expressions for different ordering
of $\sigma_1$ and $\sigma_2$. Note that \eqref{uni_max} is equivalent to say that the sequence $\{M_n-Med(M_n)\}_n$ is tight and
$$Med(M_n)=\left(\sqrt{2\log 2}\;\se\right)n-\beta \frac{\se}{\sqrt{2\log 2}}\log n +O(1),$$
where $Med(X)=\sup\{x:P(X\leq x)\leq \frac{1}{2}\}$ is the median of the random variable $X$. In the following, we will use superscripts to distinguish different cases, see \eqref{eqvar}, \eqref{inc} and \eqref{dec} below.

A special and well-known case is when $\sigma_1=\sigma_2=\sigma$, i.e.,
all the increments are i.i.d.. In that case,
the maximal displacement is described as follows:
%know to be
\begin{equation}\label{eqvar}
M_n^==\left(\sqrt{2\log 2}\;\sigma\right) n-\frac{3}{2}\frac{\sigma}{\sqrt{2\log 2}}\log n+O_P(1);
\end{equation}
the proof can be found in \cite{Addario-BerryReed09}, and its analog for branching Brownian motion can be found in \cite{Bramson78_BBM} using probabilistic techniques and \cite{Lau85} using PDE techniques. This homogeneous case corresponds to \eqref{uni_max} with $\se=\sigma$ and $\beta=\frac{3}{2}$. In this paper, we deal with the extension to the inhomogeneous case. The main results are the following two theorems.

\begin{theorem}\label{th_inc}
 When $\sigma_1^2<\sigma_2^2$ (increasing variances), the maximal displacement is
 \begin{equation}\label{inc}
 M_n^{\uparrow}=\left(\sqrt{(\sigma_1^2+\sigma_2^2)\log 2}\right)n-\frac{\sqrt{\sigma_1^2+\sigma_2^2}}{4\sqrt{\log 2}}\log n+O_P(1),
 \end{equation}
 which is of the form
 \eqref{uni_max} with $\se=\sqrt{\frac{\sigma_1^2+\sigma_2^2}{2}}$ and $\beta=\frac{1}{2}$.
\end{theorem}

\begin{theorem}\label{th_dec}
  When $\sigma_1^2>\sigma_2^2$ (decreasing variances), the maximal displacement is
  \begin{equation}\label{dec}
  M_n^{\downarrow}=\frac{\sqrt{2\log 2}(\sigma_1+\sigma_2)}{2}n-\frac{3(\sigma_1+\sigma_2)}{2\sqrt{2\log 2}}\log n+O_P(1),
  \end{equation}
  which is of the form \eqref{uni_max} with $\se=\frac{\sigma_1+\sigma_2}{2}$ and $\beta=3$.
\end{theorem}

 %This result are interesting, compared with the $2^n$ independent walks with increments %$N(0,\sigma^2)$. The maximal $(M_n^{\mbox{ind}})'$ in this independent model can be %calculated by the first and second moment methods as
%\begin{equation}\label{eqvar_ind}
%(M_n^{\text{ind}})'=\sqrt{2\sigma^2\log 2}n-\frac{\sigma^2}{2\sqrt{2\sigma^2\log %2}}\log n+O(1),\;\;\text{a.s.}.
%\end{equation}
%$M_n'$ and $(M_n^{\text{ind}})'$ possess the same law of large number, but different %$\log n$ corrections. The difference is due to the intrinsic dependence from the %branching structure in branching random walks.

%Different phenomenon occurs when $\sigma_1^2\neq \sigma_2^2$. In this case, if we %consider the corresponding independent walks model, there is no surprise that the %maximal displacement will always be

 For comparison purpose, it is useful to introduce
 the
model of
$2^n$ independent (inhomogeneous) random walks with centered independent Gaussian variables, with variance profile as
above.
%; in this model there is
%complete
%, which assumes complete spatial independence and the number of walks grows exponentially with respect to time $n$.
%Let us
Denote by $M_n^{\text{ind}}$ the maximal displacement at time $n$ in this model. Because of the complete independence, it can be easily shown that
\begin{equation}\label{ind}
M_n^{\text{ind}}=\left(\sqrt{(\sigma_1^2+\sigma_2^2)\log 2}\right)n-\frac{\sqrt{\sigma_1^2+\sigma_2^2}}{4\sqrt{\log 2}}\log n+O_P(1)
\end{equation}
for all choices of $\sigma_1^2$ and $\sigma_2^2$. Thus, in this case,
$\se=\sqrt{(\sigma_1^2+\sigma_2^2)/2}$ and $\beta=1/2$.
Thus,
the difference between $M_n^=$ and $M_n^{\text{ind}}$ when $\sigma_1^2=\sigma_2^2$ lies in the logarithmic correction. As commented (for branching Brownian motion) in \cite{Bramson78_BBM}, the different correction is due to
the intrinsic dependence between particles coming from the branching structure in branching random walks.

%which coincides with $(M_n^{\text{ind}})'$ when $\sigma_1^2=\sigma_2^2$. However, interesting behavior comes up when we consider the maximal displacement for branching random walks for $\sigma_1^2\neq \sigma_2^2$.
Another related
quantity is the sub-maximum obtained by a greedy algorithm, which only considers
the maximum over all decendents of the maximal particle at time $n/2$.
%requires every sub-step to be optimal. That is, we can consider the maximum of all particles at time $n/2$ and the branching random walks starting from that particular particle.
Applying \eqref{eqvar}, we find that the output of such algorithm
%maximum at time $n$ is at least
is
\begin{eqnarray}\label{subMax}
  &&\left(\sqrt{2\log 2}\sigma_1\frac{n}{2}-\frac{3}{2}\frac{\sigma_1}{\sqrt{2\log 2}}\log \frac{n}{2}\right)+\left(\sqrt{2\log 2}\sigma_2\frac{n}{2}-\frac{3}{2}\frac{\sigma_2}{\sqrt{2\log 2}}\log \frac{n}{2}\right)+O_P(1)\nonumber\\
  &&=\frac{\sqrt{2\log 2}(\sigma_1+\sigma_2)}{2}n-\frac{3(\sigma_1+\sigma_2)}{2\sqrt{2\log 2}}\log n+O_P(1).
\end{eqnarray}
Comparing \eqref{subMax} with the theorems, we see that
this algorithm yields the maximum up to an $O_P(1)$ error in the case
of decreasing variances (compare with
 \eqref{dec}) but not in the case of increasing variances (compare with
 \eqref{inc}) or of homogeneous increments (compare with \eqref{eqvar}).
%More importantly, together with some large deviation calculations, it will give us some intuition for finding the real maximum, which will be clear in the following chapters.

A few comparisons are now in order.
 \begin{itemize}
   \item[1.] When the variances are increasing, $M_n^{\uparrow}$ is
	   asymptotically (up to $O_P(1)$ error)
	   the same as $M_n^{\text{ind}}$, which is exactly the same as the maximum of independent homogeneous random walks with effective variance $\frac{\sigma_1^2+\sigma_2^2}{2}$.
   \item[2.] When the variances are decreasing, $M_n^{\downarrow}$ shares the same asymptotic behavior with the sub-maximum \eqref{subMax}. In this case, a greedy strategy yields the approximate maximum.
   \item[3.] With the same set of diffusivity  constants
	   $\{\sigma_1^2,\sigma_2^2\}$ but different order, $M_n^{\uparrow}$ is greater than $M_n^{\downarrow}$.
	   %An intuitive explanation is that there are more steps with large diffusivity when the diffusivity is increasing.
   \item[4.] While the leading order terms in \eqref{eqvar}, \eqref{inc} and \eqref{dec} are continuous in $\sigma_1$ and $\sigma_2$ (they coincide upon setting $\sigma_1=\sigma_2$), the logarithmic corrections exhibit
	   a phase transition phenomenon (they are not the same when we let $\sigma_1=\sigma_2$).
 \end{itemize}
%First, for the terms of order $n$,
%$M_n^{\downarrow}$ is much smaller than $M_n^{\uparrow}$ which is
%about the same as $M_n^{\text{ind}}$. An intuitive explanation for
%this is that there are more big increments in the increasing
%variance case than in the decreasing variance case. Second, the order $n$ terms %coincide when we let $\sigma_1^2=\sigma_2^2$. Third, more interesting
%are the logarithmic correction terms. $M_n^{\uparrow}$ has the same
%logarithmic correction as $M_n^{\text{ind}}$, although $M_n^{\text{ind}}$ will dominate %$M_n^{\uparrow}$ through the $O(1)$ term, while
%$M_n^{\downarrow}$ has a different logarithmic correction. Lastly, $M_n^{\uparrow}$ and %$M_n^{\downarrow}$ have different
%logarithmic corrections compared with that of $(M_n)'$ when we let
%$\sigma_1^2=\sigma_2^2$. The three do not coincide as for the $n^{\mbox{th}}$ terms. %To
%understand why, we need to look closely at the behavior of the path
%realizing the maximal displacement.

We will prove Theorem \ref{th_inc} in Section \ref{sec_inc} and
Theorem \ref{th_dec} in Section \ref{sec_dec}.
Before proving the
theorems, we state a tightness result.
%For a random variable $X$,
%let $\text{Med}(X)$ denote the median of $X$,
%$$ \text{Med}(X)=\sup\{x: P(X\leq x)\leq 1/2\}\,.$$
\begin{lemma}\label{lem_tight}
  The sequences $\{M_n^{\uparrow}-\text{Med}(M_n^{\uparrow})\}_n$ and
  $\{M_n^{\downarrow}-\text{Med}(M_n^{\downarrow})\}_n$ are tight.
\end{lemma}

This lemma follows from Bramson and Zeitouni \cite{BramsonZeitouni09}
or Fang \cite{Fang11}. One can write down a similar
recursion for the distribution of $M_n$ to the one in those two papers,
except for different subscripts and superscripts. Since the argument there depends only on one step of the recursion, it applies here directly without any change and leads to the tightness result in the lemma.

A note on notation:
throughout, we use $C$ to denote a generic positive constant,
possibly depending on $\sigma_1$ and $\sigma_2$, that may change
from line to line.
\section{Increasing Variances: $\sigma_1^2<\sigma_2^2$} \label{sec_inc}

In this section, we prove Theorem
\ref{th_inc}. We begin in Subsection \ref{sec_fluc}
with a result on the fluctuation of an
inhomogeneous random walk.
In the short
Subsection \ref{sec_ldp} we
provide large-deviations based heuristics for our results.
While it is not used in the actual proof, it
explains
the leading term of the maximal displacement and
gives hints about the
derivation
of the logarithmic correction term.
The actual proof of Theorem \ref{th_inc} is provided in subsection
\ref{sec_inc_proof}.

%In the end, we give a rigorous proof based on the intuition obtained from the large deviation analysis.

\subsection{Fluctuation of an Inhomogeneous Random Walk}\label{sec_fluc}

Let
\begin{equation}\label{inRW}
S_n=\sum_{i=1}^{n/2}X_i+\sum_{i=n/2+1}^{n}Y_i
\end{equation}
be an inhomogeneous random walk, where $X_i\sim N(0,\sigma_1^2)$, $Y_i\sim N(0,\sigma_2^2)$, and $X_i$ and $Y_i$ are independent. Define
\begin{equation}\label{intPos}
 s_{k,n}(x)=\left\{
  \begin{aligned}
  &\frac{\sigma_1^2k}{(\sigma_1^2+\sigma_2^2)\frac{n}{2}}x, &\;\;0\leq k\leq \frac{n}{2},\\
  &\frac{\sigma_1^2\frac{n}{2}+\sigma_2^2(k-\frac{n}{2})} {(\sigma_1^2+\sigma_2^2)\frac{n}{2}} x,&  \;\; \frac{n}{2}\leq k\leq n,
  \end{aligned}
  \right.
\end{equation}
and
\begin{equation}\label{bigFluc}
f_{k,n}=\left\{\begin{aligned} &c_fk^{2/3},\;&\;k\leq n/2,\\ &c_f(n-k)^{2/3},\;& \; n/2<k\leq n.\end{aligned}\right.
\end{equation}
%Consider two events
%$$A:=\left\{S_n=s_n\right\},$$
%$$B:=\left\{S_k\in[s_k-f_k,s_k+f_k],\;\text{for all}\;0\leq k\leq n\right\}.$$
As the following lemma says, conditioned on $\{S_n=x\}$, the path of the walk $S_n$ follows $s_{k,n}(x)$ with fluctuation less than or equal to $f_{k,n}$ at level $k\leq n$.

\begin{lemma}\label{lem_bigfluc}
 There exists a constant $C>0$ (independent of $n$) such that
 $$P(S_n(k)\in [s_{k,n}(S_n)-f_{k,n},s_{k,n}(S_n)+f_{k,n}]\;\text{ for all }\;0\leq k\leq n|S_n) \geq C,$$
 where $S_n(k)$ is the sum of the first $k$ summands of $S_n$, i.e.,
 $$S_n(k)=\left\{\begin{aligned} &\sum_{k=1}^{k}X_k,\;&\;k\leq n/2,\\ &\sum_{k=1}^{n/2}X_k+\sum_{k=n/2+1}^{k}Y_k,\;& \; n/2<k\leq n.\end{aligned}\right.$$
\end{lemma}

\begin{proof}
  Let $\tilde{S}_{k,n}=S_n(k)-s_{k,n}(S_n)$. Then, similar to Brownian bridge, one can check that $\tilde{S}_{k,n}$ are independent of $S_n$. To see this, first note that the covariance between $\tilde{S}_{k,n}$ and $S_n$ is
  $$Cov(\tilde{S}_{k,n},S_n)=E\tilde{S}_{k,n}S_n-E\tilde{S}_{k,n}ES_n =E\tilde{S}_{k,n}S_n,$$
  since $ES_n=0$ and $E\tilde{S}_{k,n}=0$.

  For $k\leq n/2$,
  $$\tilde{S}_{k,n}=\left(1-\frac{\sigma_1^2k}{(\sigma_1^2+\sigma_2^2)\frac{n}{2}}\right) \sum_{i=1}^{k}X_i-\frac{\sigma_1^2k}{(\sigma_1^2+\sigma_2^2)\frac{n}{2}} \sum_{i=k+1}^{n/2}X_i-\frac{\sigma_1^2k}{(\sigma_1^2+\sigma_2^2)\frac{n}{2}} \sum_{i=n/2+1}^{n}Y_i.$$
  Expand $\tilde{S}_{k,n}S_n$, take expectation, and then all terms vanish except for those containing $X_i^2$ and $Y_i^2$. Taking into account that $EX_i^2=\sigma_1^2$ and $EY_i^2=\sigma_2^2$, one has
  \begin{eqnarray}\label{cov}
  &&Cov(\tilde{S}_{k,n},S_n)=E\tilde{S}_{k,n}S_n\nonumber\\  &=&\left(1-\frac{\sigma_1^2k}{(\sigma_1^2+\sigma_2^2)\frac{n}{2}}\right) \sum_{i=1}^{k}EX_i^2 -\frac{\sigma_1^2k}{(\sigma_1^2+\sigma_2^2)\frac{n}{2}} \sum_{i=k+1}^{n/2}EX_i^2 -\frac{\sigma_1^2k}{(\sigma_1^2+\sigma_2^2)\frac{n}{2}} \sum_{i=n/2+1}^{n}EY_i^2 \nonumber\\
  &=&\left(1-\frac{\sigma_1^2k}{(\sigma_1^2+\sigma_2^2)\frac{n}{2}}\right)k\sigma_1^2 -\frac{\sigma_1^2k}{(\sigma_1^2+\sigma_2^2)\frac{n}{2}} (n/2-k)\sigma_1^2 -\frac{\sigma_1^2k}{(\sigma_1^2+\sigma_2^2)\frac{n}{2}} (n/2)\sigma_2^2\nonumber\\
  &=&0.
  \end{eqnarray}

  For $n/2< k\leq n$, one can calculate $Cov(\tilde{S}_{k,n},S_n)=0$ similarly as follows. First,
  $$\tilde{S}_{k,n}=\frac{\sigma_2^2(n-k)}{(\sigma_1^2+\sigma_2^2)\frac{n}{2}} \sum_{i=1}^{n/2}X_i+\frac{\sigma_2^2(n-k)}{(\sigma_1^2+\sigma_2^2)\frac{n}{2}} \sum_{i=n/2+1}^{k}Y_i- \left(1-\frac{\sigma_2^2(n-k)}{(\sigma_1^2+\sigma_2^2)\frac{n}{2}}\right) \sum_{i=k+1}^{n}Y_i.$$
  Then, expanding $\tilde{S}_{k,n}S_n$ and taking expectation, one has
  \begin{eqnarray*}
    &&Cov(\tilde{S}_{k,n},S_n)=E\tilde{S}_{k,n}S_n\nonumber\\
    &=& \frac{\sigma_2^2(n-k)}{(\sigma_1^2+\sigma_2^2)\frac{n}{2}} \sum_{i=1}^{n/2}EX_i^2+\frac{\sigma_2^2(n-k)}{(\sigma_1^2+\sigma_2^2)\frac{n}{2}} \sum_{i=n/2+1}^{k}EY_i^2- \left(1-\frac{\sigma_2^2(n-k)}{(\sigma_1^2+\sigma_2^2)\frac{n}{2}}\right) \sum_{i=k+1}^{n}EY_i^2\\
    &=& \frac{\sigma_2^2(n-k)}{(\sigma_1^2+\sigma_2^2)\frac{n}{2}} (n/2)\sigma_1^2+\frac{\sigma_2^2(n-k)}{(\sigma_1^2+\sigma_2^2)\frac{n}{2}} (k-n/2)\sigma_2^2- \left(1-\frac{\sigma_2^2(n-k)}{(\sigma_1^2+\sigma_2^2)\frac{n}{2}}\right) (n-k)\sigma_2^2\\
    &=&0
  \end{eqnarray*}

  Therefore, $\tilde{S}_{k,n}$ are independent of $S_n$ since they are Gaussian. Using this independence,
%  $A=\{S_n=s_n\}$ and $B=\{S_n=s_n,\tilde{S}_k\in[-f_k,f_k],\;\text{for all}\;0\leq %k\leq n\}$. So
%  $$P(B)=P(A)P(\tilde{S}_k\in[-f_k,f_k],\;\text{for all}\;0\leq k\leq n).$$
  \begin{eqnarray*}
  &&P(S_n(k)\in [s_{k,n}(S_n)-f_{k,n},s_{k,n}(S_n)+f_{k,n}]\;\text{ for all }\;0\leq k\leq n|S_n)\\
  &=&P(\tilde{S}_{k,n}\in [-f_{k,n},f_{k,n}]\;\text{ for all }\;0\leq k\leq n|S_n)\\
  &=&P(\tilde{S}_{k,n}\in [-f_{k,n},f_{k,n}]\;\text{ for all }\;0\leq k\leq n).
  \end{eqnarray*}
  By calculation similar to \eqref{cov}, $\tilde{S}_{k,n}$ is a Gaussian sequence with mean zero and variance $k\sigma_1^2\frac{\left((\sigma_1^2+\sigma_2^2)n-2\sigma_1^2k\right)} {(\sigma_1^2+\sigma_2^2)n}$ for $k\leq n/2$ and $(n-k)\sigma_2^2 \frac{\left((\sigma_1^2+\sigma_2)n-2\sigma_2^2(n-k)\right)}{(\sigma_1^2+\sigma_2^2)n}$ for $n/2<k\leq n$. The above quantity is
  $$ 1-P(|\tilde{S}_{k,n}|>f_{k,n},\;\text{for some}\;0\leq k\leq n)\geq 1-\sum_{k=1}^{n}P(|\tilde{S}_{k,n}|>f_{k,n}).$$
  Using a standard Gaussian estimate, e.g. \cite[Theorem 1.4]{Durrett05}, the above quantity is at least,
    $$1-\sum_{k=1}^{n} \frac{c_0}{\sqrt{k}}e^{-\frac{f_{k,n}^2}{k}c_1} \geq 1-2\sum_{k=1}^{\infty}\frac{c_0}{\sqrt{k}}e^{-c_f^2c_1k^{1/3}}:=C>0$$
  where $c_0,c_1$ are constants depending on $\sigma_1$ and $\sigma_2$, and $C>0$ can be realized by choosing the constant $c_f$ large. This proves the lemma.
\end{proof}

\subsection{Sample Path Large Deviation Heuristics}\label{sec_ldp}
We explain (without giving a proof) what we expect for the order $n$ term of $M_n{\uparrow}$, by giving a large
deviation argument. The exact proof will be postponed to the next subsection. Consider the same $S_n$ as defined in \eqref{inRW} and a function $\phi(t)$ defined on $[0,1]$ with $\phi(0)=0$. A sample path large deviation result, see \cite[Theorem 5.1.2]{DemboZeitouni98_LDP}, tells us that the probability for $S_{\lfloor rn\rfloor}$ to be roughly $\phi(r)n$ for $0\leq r\leq s\leq 1$ is roughly $e^{-nI_s(\phi)}$, where
\begin{equation}\label{rateFun}
  I_s(\phi)=\int_0^s\Lambda^*_r(\dot{\phi}(r))dr,
\end{equation}
$\dot{\phi}(r)=\frac{d}{dr}\phi(r)$, and $\Lambda^*_r(x)=\frac{x^2}{2\sigma_1^2}$, for $0\leq r\leq 1/2$, and $\frac{x^2}{2\sigma_1^2}$, for $1/2<r\leq 1$. A first moment argument would yield a necessary condition for a walk that roughly follows the path $\phi(r)n$ to exist among the branching random walks,
\begin{equation}\label{LDPshort}
  I_s(\phi)\leq s\log 2,\;\;\mbox{ for all }\; 0\leq s\leq 1.
\end{equation}
This is equivalent to
\begin{equation}\label{LDP}
  \left\{
  \begin{aligned}
  &\int_0^s \frac{\dot{\phi}^2(r)}{2\sigma_1^2}dr\leq s\log 2, &\;\;0\leq s\leq \frac{1}{2},\\
  &\int_0^{\frac{1}{2}}\frac{\dot{\phi}^2(r)}{2\sigma_1^2}dr+\int_{\frac{1}{2}}^s \frac{\dot{\phi}^2(r)}{2\sigma_2^2}dr\leq s\log 2,&  \;\; \frac{1}{2}\leq s\leq 1.
  \end{aligned}
  \right.
\end{equation}
Otherwise, if \eqref{LDPshort} is violated for some $s_0$, i.e., $I_{s_0}(\phi)> s_0\log 2$, there will be no path following $\phi(r)n$ to $\phi(s_0)n$, since the expected number of such paths is $2^{sn}e^{-nI_s(\phi)}=e^{-(I_s(\phi)-s\log 2)n}$, which decreases exponentially.

Our goal is then to maximize $\phi(1)$ under the constraints \eqref{LDP}. By Jensen's inequality and convexity, one can prove that it is equivalent to maximizing $\phi(1)$ subject to
\begin{equation}\label{optimization}
    \frac{\phi^2(1/2)}{\sigma_1^2}\leq \frac{1}{2} \log 2,\; \frac{\phi^2(1/2)}{\sigma_1^2}+\frac{(\phi(1)-\phi(1/2))^2}{\sigma_2^2}\leq \log 2.
\end{equation}
Note that the above argument does not necessarily require $\sigma_1^2<\sigma_2^2$.

Under the assumption that $\sigma_1^2<\sigma_2^2$, we can solve the optimization problem with the optimal curve
\begin{equation}\label{inc_LDP_curve}
  \phi(s)=\left\{
  \begin{aligned}
  &\frac{2\sigma_1^2\sqrt{\log 2}}{\sqrt{(\sigma_1^2+\sigma_2^2)}}s, &\;\;0\leq s\leq \frac{1}{2},\\
  &\frac{2\sigma_1^2\sqrt{\log 2}}{\sqrt{(\sigma_1^2+\sigma_2^2)}}\frac{1}{2}+\frac{2\sigma_2^2\sqrt{\log 2}}{\sqrt{(\sigma_1^2+\sigma_2^2)}}(s-\frac{1}{2}),&  \;\; \frac{1}{2}\leq s\leq 1.
  \end{aligned}
  \right.
\end{equation}

If we plot this optimal curve and the suboptimal curve leading to \eqref{subMax} as in Figure \ref{incFig}, it is easy to see that the ancestor at time $n/2$ of the actual maximum at time $n$ is not a maximum at time $n/2$, since $\frac{2\sigma_1^2\sqrt{\log 2}}{\sqrt{(\sigma_1^2+\sigma_2^2)}}<\sqrt{2\sigma_1^2\log 2}$. A further rigorous calculation as in the next subsection shows that, along the optimal curve \eqref{inc_LDP_curve}, the branching random walks have an exponential decay of correlation. Thus a fluctuation between $n^{1/2}$ and $n$ that is larger than the typical fluctuation of a random walk is admissible. This is consistent with the naive observation from Figure \ref{incFig}. This kind of behavior also occurs in the independent random walks model, explaining why $M_n^{\uparrow}$ and $M_n^{\mbox{ind}}$ have the same asymptotical expansion up to an $O(1)$ error, see \eqref{inc} and \eqref{ind}.
%The maximal displacement at time $\frac{n}{2}$ is $\sqrt{2\sigma_1^2\log 2}\frac{n}{2}+o(n)$, and the maximal displacement at time $n$ of the branching random walk starting from the maxima at time $\frac{n}{2}$ is $\sqrt{2\sigma_1^2\log 2}\frac{n}{2}+\sqrt{2\sigma_2^2\log 2}\frac{n}{2}+O(n)$. If we plot this curve with $\phi(t)$ as in Figure \ref{fig_inc}, the graph indicates that there are exponentially many particles along the $\phi(t)$. This gives an idea of why $M_n^{\uparrow}$ and $M_n^{\text{ind}}$ have the same asymptotic behavior. To actually prove Theorem \ref{th_inc}. From the graph, it seems that the paths leading to the maximal may have a fluctuation between $n^{1/2}$ and $n$. For example, let us consider fluctuation $n^{2/3}$ and introduce the following lemma.

\begin{figure}[h]
  \centering
  \includegraphics[scale=0.8]{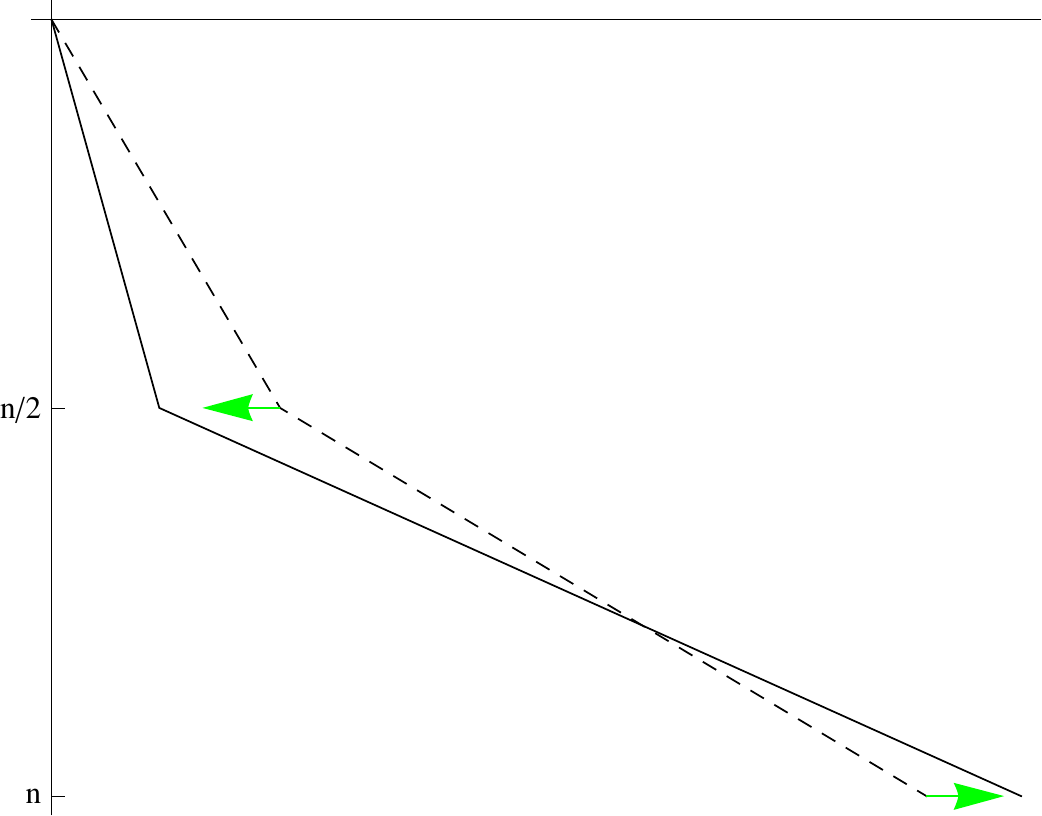}
  \caption{Dashed: maximum at time $n$ of BRW starting from maximum at time $n/2$.\newline \quad Solid: maximum at time $n$ of BRW starting from time $0$.}\label{incFig}
\end{figure}

\subsection{Proof of Theorem \ref{th_inc}}\label{sec_inc_proof}
With Lemma \ref{lem_bigfluc} and the observation from Section \ref{sec_ldp}, we can now provide a proof of Theorem \ref{th_inc}, applying the first and second moments method to the appropriate sets. In the proof, we use $S_n$ to denote the walk defined by \eqref{inRW} and $S_k$ to denote the sum of the first $k$ summand in $S_n$.
\begin{proof}[Proof of Theorem \ref{th_inc}]
\quad {\it Upper bound.} Let $a_n=\sqrt{(\sigma_1^2+\sigma_2^2)\log 2}n-\frac{\sqrt{\sigma_1^2+\sigma_2^2}}{4\sqrt{\log 2}}\log n$. Let $N_{1,n}=\sum_{v\in\mathds{D}_n}1_{\{S_v>a_n+y\}}$ be the number of particles at time n whose displacements are greater than $a_n+y$. Then
$$EN_{1,n}=2^nP(S_n\geq a_n+y)\leq c_2e^{-c_3y}$$
where $c_2$ and $c_3$ are constants independent of $n$ and the last inequality is due to the fact that $S_n\sim N(0,\frac{\sigma_1^2+\sigma_2^2}{2}n)$. So we have, by the Chebyshev's inequality,
\begin{equation}\label{incUB}
P(M_n^{\uparrow}>a_n+y)=P(N_1\geq 1)\leq EN_{1,n}\leq c_2e^{-c_3y}.
\end{equation}
Therefore, this probability can be made as small as we wish by choosing a large $y$.

{\it Lower bound.} Consider the walks which are at $s_n\in I_n=[a_n,a_n+1]$ at time n and follow $s_{k,n}(s_n)$, defined by \eqref{intPos}, at intermediate times with fluctuation bounded by $f_{k,n}$, defined by \eqref{bigFluc}. Let $I_{k,n}(x)=[s_{k,n}(x)-f_{k,n},s_{k,n}(x)+f_{k,n}]$ be the `admissible' interval at time $k$ given $S_n=x$, and let
$$N_{2,n}=\sum_{v\in\mathds{D}_n}1_{\{S_v\in I_n,S_{v^k}\in I_{k,n}(S_v) \mbox{ for all } 0\leq k\leq n\}}$$
be the number of such walks. By Lemma \ref{lem_bigfluc},
\begin{eqnarray}\label{incLBfirst}
EN_{2,n}&=&2^nP(S_n\in I_n,S_n(k)\in I_{k,n}(S_n) \mbox{ for all } 0\leq k\leq n ) \nonumber\\
&=&2^nE(1_{\{S_n\in I_n\}}P(S_n(k)\in I_{k,n}(S_n) \mbox{ for all } 0\leq k\leq n |S_n))\nonumber \\
&\geq &2^nCP(S_n\in I_n)\geq c_4.
\end{eqnarray}

Next, we bound the second moment $EN_{2,n}^2$. By considering the location of any pair $v_1,v_2\in\mathds{D}_n$ of particles at time $n$ and at their common ancestor $v_1\wedge v_2$, we have
\begin{eqnarray*}
  &&EN_{2,n}^2=E\sum_{v_1,v_2\in\mathds{D}_n}1_{\{S_{v_i}\in I_n,\;S_{(v_i)^j}\in I_{j,n}(S_{(v_i)^j}) \mbox{ for all } 0\leq j\leq n,i=1,2\}}\\
  &=& \sum_{k=0}^n\sum_{\substack{v_1,v_2\in\mathds{D}_n\\ v_1\wedge v_2\in \mathds{D}_k}}E1_{\{S_{v_i}\in I_n,\;S_{(v_i)^j}\in I_{j,n}(S_{(v_i)^j}) \mbox{ for all } 0\leq j\leq n,i=1,2\}}\\
  &\leq & \sum_{k=0}^n\sum_{\substack{v_1,v_2\in\mathds{D}_n\\ v_1\wedge v_2\in \mathds{D}_k}}P(S_{v_1}\in I_n,\;S_{(v_1)^j}\in I_{j,n}(S_{(v_1)^j}) \mbox{ for all } 0\leq j\leq n)\\
  &&\;\;\;\;\;\;\;\;\;\;\;\;\;\;\;\;\;\cdot P(S_{v_2}-S_{v_1\wedge v_2}\in [x-s_{k,n}(x)-f_{k,n},x-s_{k,n}(x)+f_{k,n}],x\in I_n),
\end{eqnarray*}
where we use the independence between $S_{v_2}-S_{v_1\wedge v_2}$ and $S_{(v_1)^j}$ in the last inequality. And the last expression (double sum) in the above display is
\begin{eqnarray*}
  && \sum_{k=0}^n2^{2n-k}P(S_n\in I_n,S_n(j)\in I_{j,n}(S_n) \mbox{ for all } 0\leq j\leq n )\\
  &&\;\;\;\;\;\;\;\cdot P(S_n-S_n(k) \in[x-s_{k,n}(x)-f_{k,n},x-s_{k,n}(x)+f_{k,n}],x\in I_n)\\
  &\leq& EN_{2,n}\sum_{k=0}^n2^{n-k}P(S_n-S_n(k) \in[x-s_{k,n}(x)-f_{k,n},x-s_{k,n}(x)+f_{k,n}],x\in I_n).
\end{eqnarray*}
The above probabilities can be estimated separately when $k\leq n/2$ and $n/2<k\leq n$. For $k\leq n/2$, $S_n-S_n(k)\sim N(0,\frac{n}{2}(\sigma_1^2+\sigma_2^2)-k\sigma_1^2)$. Thus,
\begin{eqnarray*}
 && P(S_n-S_n(k) \in[x-s_{k,n}(x)-f_{k,n},x-s_{k,n}(x)+f_{k,n}],x\in I_n)\\
 &\leq & 2f_{k,n}\frac{1}{\sqrt{\pi ((\sigma_1^2+\sigma_2^2)n-2k\sigma_1^2)}} \exp\left(-\frac{\left((1-\frac{2\sigma_1^2k}{(\sigma_1^2+\sigma_2^2)n})a_n-f_{k,n} \right)^2} {(\sigma_1^2+\sigma_2^2)n-2k\sigma_1^2}\right)\\
 &\leq & 2^{-n+\frac{2\sigma_1^2}{\sigma_1^2+\sigma_2^2}k+o(k)}.
\end{eqnarray*}
For $n/2<k\leq n$, $S_n-S_n(k)\sim N(0,(n-k)\sigma_2^2)$. Thus,
\begin{eqnarray*}
 &&  P(S_n-S_n(k) \in[x-s_{k,n}(x)-f_{k,n},x-s_{k,n}(x)+f_{k,n}],x\in I_n)\\
 &\leq & 2f_{k,n}\frac{1}{\sqrt{2\pi(n-k)\sigma_2^2}} \exp\left(-\frac{\left(\frac{2\sigma_2^2(n-k)} {(\sigma_1^2+\sigma_2^2)n}a_n-f_{k,n}\right)^2} {2(n-k)\sigma_2^2}\right)\\
 &\leq & 2^{-\frac{2\sigma_2^2}{\sigma_1^2+\sigma_2^2}(n-k)+o(n-k)}.
\end{eqnarray*}
Therefore,
\begin{equation}\label{incLBsec}
EN_{2,n}^2\leq EN_{2,n}\left(\sum_{k=0}^{n/2}2^{\frac{\sigma_1^2-\sigma_2^2}{\sigma_1^2+\sigma_2^2} k+o(k)}+\sum_{k=n/2+1}^n2^{\frac{\sigma_1^2-\sigma_2^2}{\sigma_1^2+\sigma_2^2}(n-k) +o(n-k)}\right)\leq c_5EN_{2,n},
\end{equation}
where $c_5=2\sum_{k=0}^{\infty}2^{\frac{\sigma_1^2-\sigma_2^2}{\sigma_1^2+\sigma_2^2} k+o(k)}$.
By the Cauchy-Schwartz inequality,
\begin{equation}\label{incLB}
P(M_n^{\uparrow}\geq a_n)\geq P(N_{2,n}>0)\geq \frac{(EN_{2,n})^2}{EN_{2,n}^2}\geq c_4/c_5>0.
\end{equation}

The upper bound \eqref{incUB} and lower bound \eqref{incLB} imply that there exists a large enough constant $y_0$ such that
$$P(M_n^{\uparrow}\in [a_n,a_n+y_0])\geq \frac{c_4}{2c_5}>0.$$
Lemma \ref{lem_tight} tells us that the sequence $\{M_n^{\uparrow}-\text{Med}(M_n^{\uparrow})\}_n$ is tight, so $M_n^{\uparrow}=a_n+O(1)$ a.s.. That completes the proof.
\end{proof}

\section{Decreasing Variances: $\sigma_1^2>\sigma_2^2$}\label{sec_dec}

We will again separate the proof of  Theorem \ref{th_dec} into two parts, the lower bound and the upper bound. Fortunately, we can apply \eqref{eqvar} directly to get a lower bound so that we can avoid repeating the second moment argument. However, we do need to reproduce (the first moment argument) part of the proof of \eqref{eqvar} in order to get an upper bound.

\subsection{An Estimate for Brownian Bridge}
Toward this end, we need the following analog of Bramson \cite[Proposition 1']{Bramson78_BBM}. The original proof in Bramson's used the Gaussian density and reflection principle of continuous time Brownian motion, which also hold for the discrete time version. The proof extends without much effort to yield the following estimate for the Brownian bridge $B_k-\frac{k}{n}B_n$, where $B_n$ is a random walk with standard normal increments.

\begin{lemma}\label{lem_bb}
 Let
   $$L(k)=\left\{\begin{array}{ll}
               0 & \mbox{ if } s=0,n, \\
               100\log k & \mbox{ if } k=1,\dots,n/2,\\
               100\log (n-k) & \mbox{ if } k=n/2,\dots,n-1.
   \end{array}\right.$$
  Then, there exists a constant $C$ such that, for all $y>0$,
  $$P(B_n-\frac{k}{n}B_n\leq L(k)+y \mbox{ for } 0\leq k \leq n)\leq \frac{C(1+y)^2}{n}.$$
\end{lemma}

The coefficient $100$ before $\log$ is chosen large enough to be suitable for later use, and is not crucial in Lemma \ref{lem_bb}.

\subsection{Proof of Theorem \ref{th_dec}}
Before proving the theorem, we discuss the equivalent optimization problems \eqref{LDP} and \eqref{optimization} under our current setting $\sigma_1^2>\sigma_2^2$. It can be solved by employing the optimal curve
\begin{equation}\label{dec_LDP_curve}
  \phi(s)=\left\{
  \begin{aligned}
  &\sqrt{2\log 2}\sigma_1s, &\;\;0\leq s\leq \frac{1}{2},\\
  &\sqrt{2\log 2}\sigma_1\frac{1}{2}+\sqrt{2\log 2}\sigma_2(s-\frac{1}{2}),&  \;\; \frac{1}{2}\leq s\leq 1.
  \end{aligned}
  \right.
\end{equation}

If we plot the curve $\phi(s)$ and the suboptimal curve leading to \eqref{subMax} as in Figure \ref{decFig}, these two curves coincide with each other up to order $n$. Figure \ref{decFig} seems to indicate that the maximum at time $n$ for the branching random walk starting from time $0$ comes from the maximum at time $n/2$. As will be shown rigorously, if a particle at time $n/2$ is left significantly behind the maximum, its descendents will not be able to catch up by time $n$. The difference between Figure \ref{incFig} and Figure \ref{decFig} explains the difference in the logarithmic correction between $M_n^{\uparrow}$ and $M_n^{\downarrow}$.

\begin{figure}[h]
  \centering
  \includegraphics[scale=0.8]{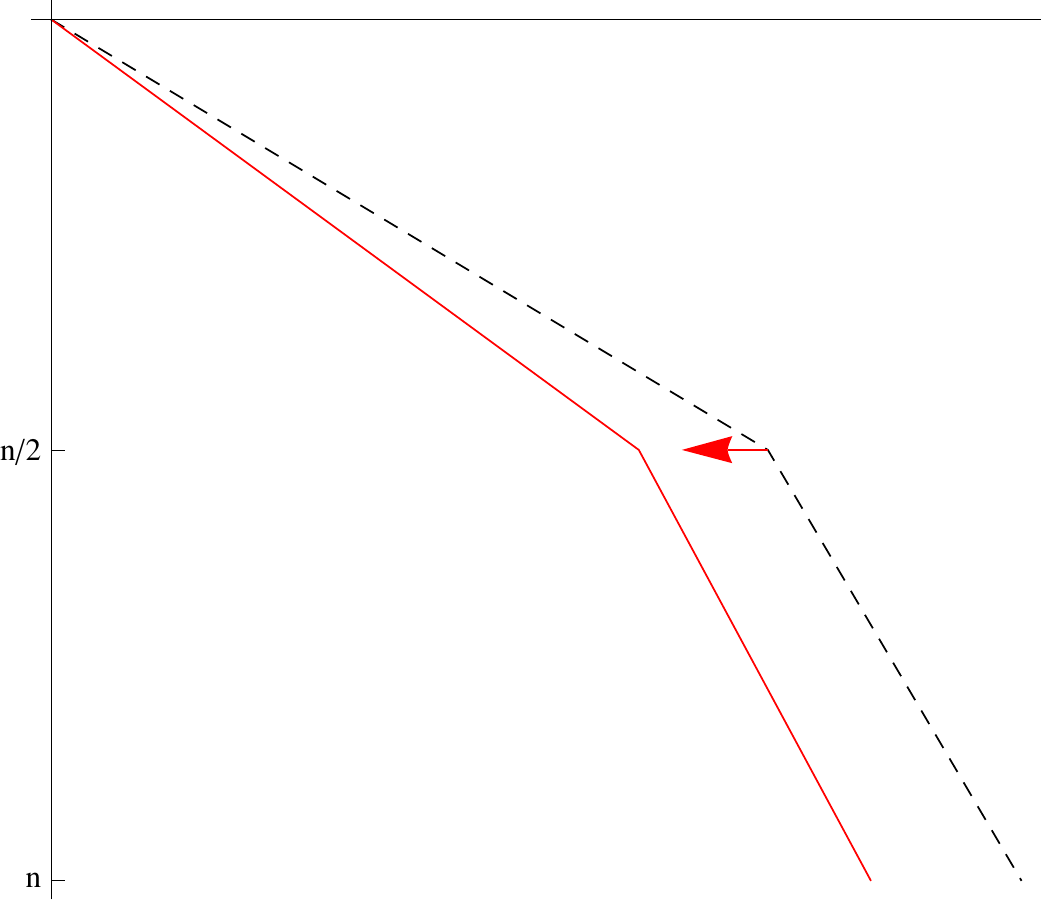}
  \caption{Dash: both the optimal path to the maximum at time $n$ and the path leading to the maximum of BRW starting from the maximum at time $n/2$. Solid: the path to the maximal (rightmost) descendent of a particle at time $n/2$ that is significantly behind the maximum then.}\label{decFig}
\end{figure}

%\textbf{Claim}\quad $M_n=(\sqrt{2\log 2}+\sqrt{\log %2})\frac{n}{2}-\frac{3}{2}(\frac{1}{\sqrt{2\log 2}}+\frac{1}{2\sqrt{\log 2}})\log %n+O(1)$.

\begin{proof}[Proof of Theorem \ref{th_dec}]
  {\it Lower Bound.} For each $i=1,2$,  the formula \eqref{eqvar} implies that there exist $y_i$ (possibly negative) such that, for branching random walk at time $n/2$ with variance $\sigma_i^2$,
  $$P\left(M_{n/2}>\sqrt{2\log 2}\sigma_i\frac{n}{2}-\frac{3\sigma_i}{2\sqrt{2\log 2}}\log \frac{n}{2}+y_i\right)\geq \frac{1}{2}.$$
  By considering a branching random walk starting from a particle at time $n/2$, whose location is greater than $\sqrt{2\log 2}\sigma_1\frac{n}{2}-\frac{3\sigma_1}{2\sqrt{2\log 2}}\log \frac{n}{2}+y_1$, and applying the above display with $i=1$ and $2$,we know that
  \begin{equation}\label{decLB}
  P\left(M_n^{\downarrow}>\frac{\sqrt{2\log 2}(\sigma_1+\sigma_2)}{2}n-\frac{3(\sigma_1+\sigma_2)}{2\sqrt{2\log 2}}\log \frac{n}{2}+y_1+y_2\right)\geq \frac{1}{4}.
 \end{equation}

  {\it Upper Bound.}  We will use a first moment argument to prove that there exists a constant $y$ (large enough) such that
  \begin{equation}\label{decUB}
  P\left(M_n^{\downarrow}>\frac{\sqrt{2\log 2}(\sigma_1+\sigma_2)}{2}n-\frac{3(\sigma_1+\sigma_2)}{2\sqrt{2\log 2}}\log \frac{n}{2}+y\right)<\frac{1}{10}.
  \end{equation}
Similarly to the last argument in the proof of Theorem \ref{th_inc}, the upper bound \eqref{decUB} and the lower bound \eqref{decLB}, together with the tightness result from Lemma \ref{lem_tight}, prove Theorem \ref{th_dec}. So it remains to show \eqref{decUB}.

Toward this end, we define a polygonal line (piecewise linear curve) leading to $\frac{\sqrt{2\log 2}(\sigma_1+\sigma_2)}{2}n-\frac{3(\sigma_1+\sigma_2)}{2\sqrt{2\log 2}}\log \frac{n}{2}$ as follows: for $1\leq k\leq n/2$,
  $$M(k)= \frac{k}{n/2}(\sqrt{2\log 2}\sigma_1\frac{n}{2}
                         -\frac{3\sigma_1}{2\sqrt{2\log 2}}\log \frac{n}{2});$$
and for $n/2+1\leq k\leq n$,
$$M(k)=M(n/2)+ \frac{k-n/2}{n/2}(\sqrt{2\log 2}\sigma_2\frac{n}{2}
                         -\frac{3\sigma_2}{2\sqrt{2\log 2}}\log \frac{n}{2}).$$
Note that $\frac{k}{n}\log n\leq \log k$ for $k\leq n$. Also define
  $$f(k)=\left\{\begin{array}{ll}
                          y & k=0,\frac{n}{2},n,\\
                          y+\frac{5\sigma_1}{2\sqrt{2\log 2}}\log k & 1\leq k\leq n/4, \\
                          y+\frac{5\sigma_1}{2\sqrt{2\log 2}}\log (\frac{n}{2}-k) & \frac{n}{4}\leq k\leq \frac{n}{2}-1, \\
                          y+\frac{5\sigma_2}{2\sqrt{2\log 2}}\log (k-\frac{n}{2}) & \frac{n}{2}+1\leq k\leq \frac{3n}{4},\\
                          y+\frac{5\sigma_2}{2\sqrt{2\log 2}}\log (n-k) & \frac{3n}{4}\leq k\leq n-1.
                        \end{array}\right.$$
We will use $f(k)$ to denote the allowed offset (deviation) from $M(k)$ in the following argument.

The probability on the left side of \eqref{decUB} is equal to
  $$P(\exists v\in\mathds{D}_n \mbox{ such that } S_v>M(n)+y).$$
For each $v\in\mathds{D}_n$, we define $\tau_v=\inf\{k:S_{v^k}>M(k)+f(k)\}$; then \eqref{decUB} is implied by
  \begin{equation}\label{decUBsum}
  \sum_{k=1}^nP(\exists v\in\mathds{D}_n \mbox{ such that } S_v>M(n)+y,\tau_v=k)<1/10.
  \end{equation}
We will split the sum into four regimes: $[1,n/4]$, $[n/4,n/2]$, $[n/2,3n/4]$ and $[3n/4,n]$, corresponding to the four parts of the definition of $f(k)$. The sum over each regime, corresponding to the events in the four pictures in Figure \ref{figFour}, can be made small. The first two are the discrete analog of the upper bound argument in Bramson \cite{Bramson78_BBM}. We will present a complete proof for the first two cases, since the argument is not too long and the argument (not only the result) is used in the latter two cases.
  \begin{figure}[h]
   \centering
    \subfigure[1 to n/4]
    {\includegraphics[scale=0.2]{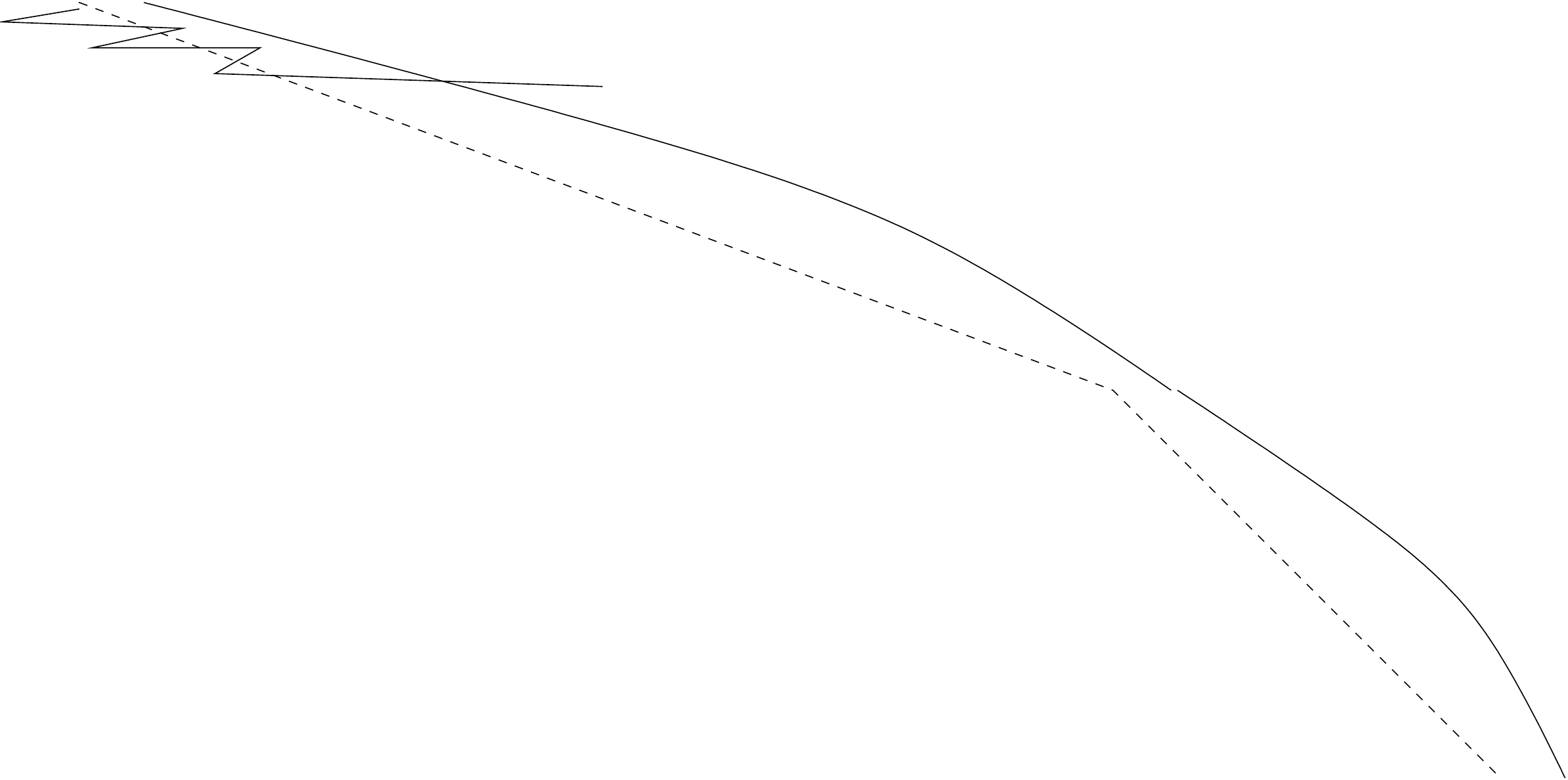}   }
    \subfigure[n/4 to n/2]
    {\includegraphics[scale=0.2]{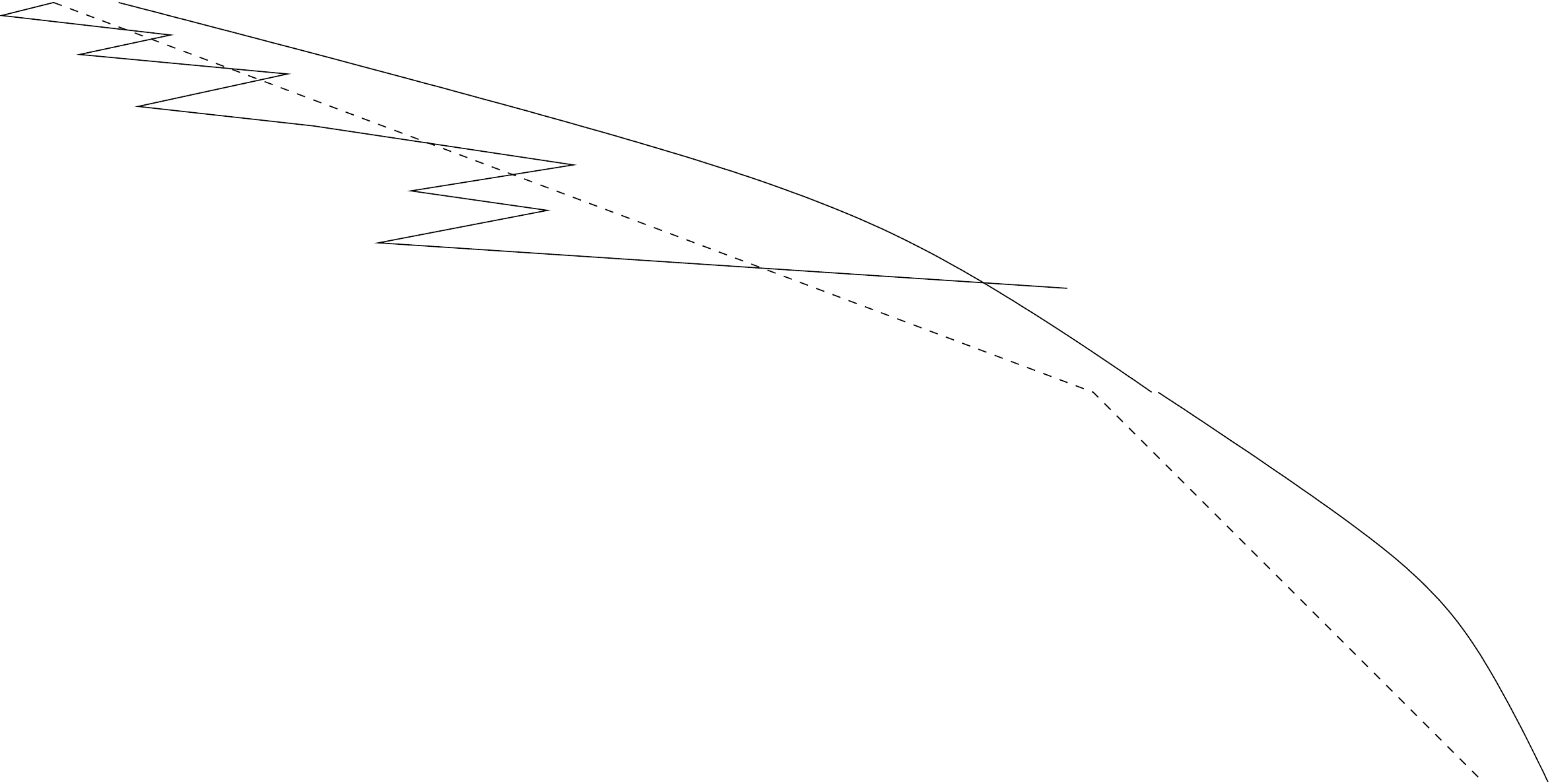}   }
    \\
    \subfigure[n/2 to 3n/4]
    {\includegraphics[scale=0.2]{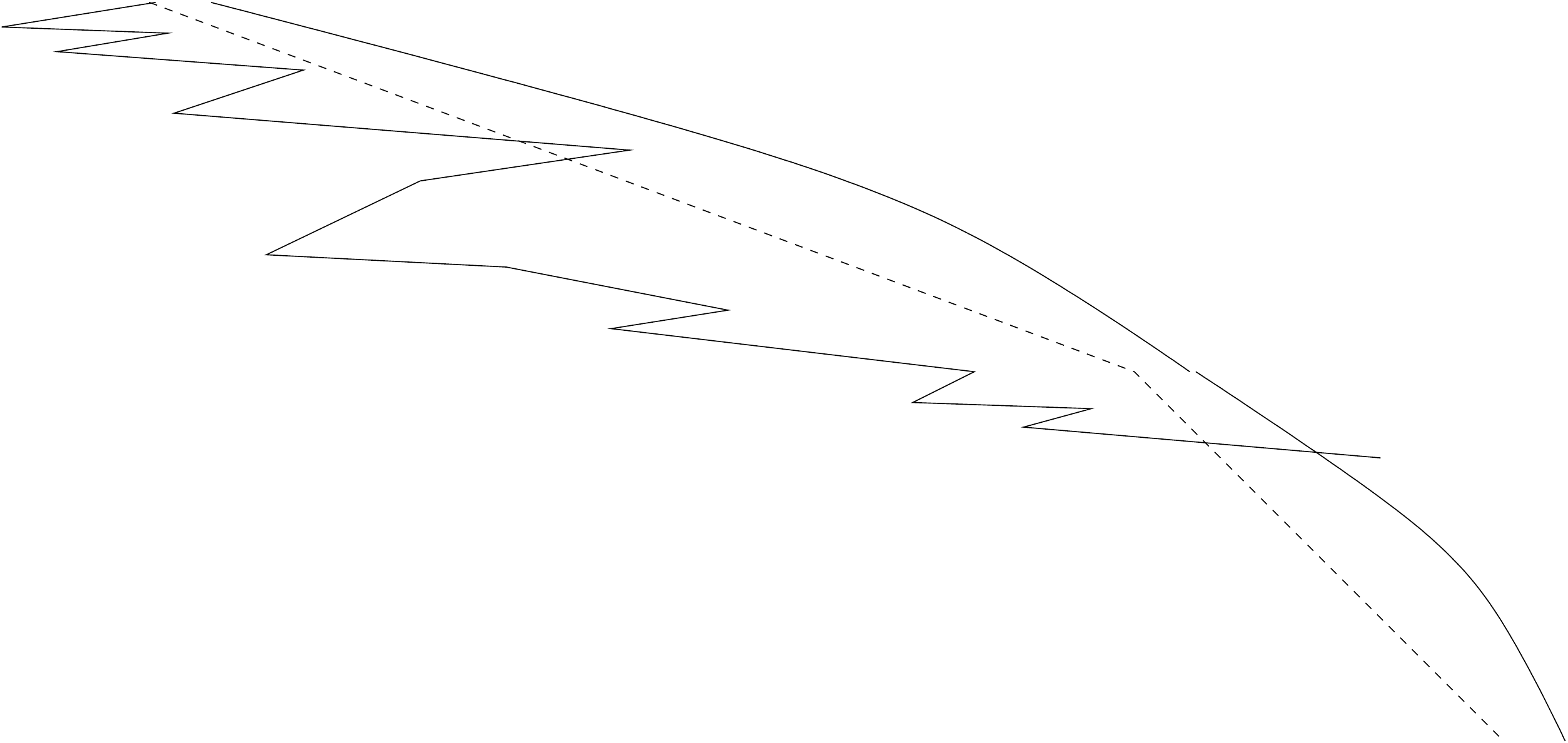}   }
    \subfigure[3n/4 to n]
    {\includegraphics[scale=0.2]{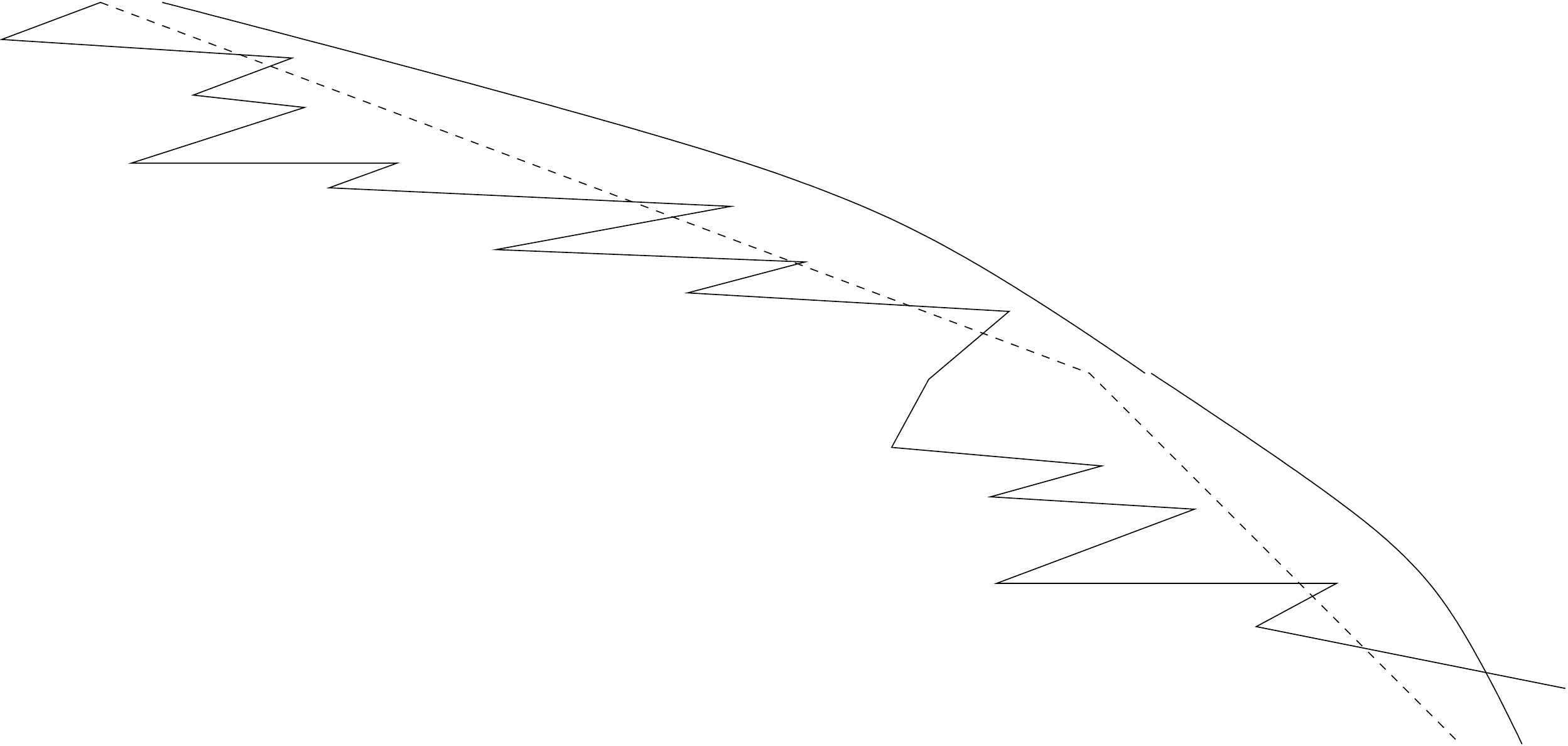} }
   \caption{Four small probability events. Dash line: $M(k)$. Solid curve: $M(k)+f(k)$.\newline Polygonal line: a random walk.}
   \label{figFour}
  \end{figure}

(\romannumeral1). When $1\leq k \leq n/4$, we have, by the Chebyshev's inequality,
        \begin{eqnarray*}
             &&P(\exists v\in\mathds{D}_n \mbox{ such that } S_v>M(n)+y,\tau_v=k)\\
             &\leq & P(\exists v\in\mathds{D}_k,\mbox{ such that } S_v>M(k)+f(k))\leq
              E\sum_{v\in\mathds{D}_k}1_{\{S_v>M(k)+f(k)\}}.
        \end{eqnarray*}
     The above expectation is less than or equal to
        \begin{eqnarray}\label{endptEst}
          \frac{C2^k} {\sqrt{k}}e^{-\frac{(M(k)+f(k))^2}{2\sigma_1^2}} &\leq &  \frac{C2^k}{\sqrt{k}}\exp\left(-\frac{\left(\sqrt{2\log 2}\sigma_1k +\frac{\sigma_1}{\sqrt{2\log 2}}\log k+y\right)^2}{2k\sigma_1^2}\right)\nonumber\\
           &\leq & Ck^{-3/2} e^{-\frac{\sqrt{2\log 2}}{\sigma_1}y}.
        \end{eqnarray}
     Summing these upper bounds over $k\in[1,n/4]$, we obtain that
     \begin{equation}\label{quad1}
     \sum_{k=1}^{n/4}P(\exists v\in\mathds{D}_n \mbox{ such that } S_v>M(n)+y,\tau_v=k)
     \leq Ce^{-\frac{\sqrt{2\log 2}}{\sigma_1}y}\sum_{k=1}^{\infty} k^{-3/2}.
     \end{equation}
     The right side of the above inequality can be made as small as we wish, say at most $\frac{1}{100}$, by choosing $y$ large enough.

(\romannumeral2).  When $n/4\leq k\leq n/2$, we again have, by Chebyshev's inequality,
        \begin{eqnarray*}
             &&P(\exists v\in\mathds{D}_n \mbox{ such that } S_v>M(n)+y,\tau_v=k)\\
             &\leq & P(\exists v\in\mathds{D}_k,\mbox{ such that } S_v>M(k)+f(k),
             \mbox{ and } S_{v^i}\leq M(i)+f(i)\mbox{ for }1\leq i\leq k)\\
             &\leq &
              E\sum_{v\in\mathds{D}_k}1_{\{S_v>M(k)+f(k), \mbox{ and } S_{v^i}\leq M(i)+f(i)\mbox{ for }1\leq i< k\}}.
        \end{eqnarray*}
      Letting $S_k$ be a copy of the random walks before time $n/2$, then the above expectation is equal to
      \begin{eqnarray}\label{quad2Prob}
        &&2^kP(S_k>M(k)+f(k), \mbox{ and } S_i\leq M(i)+f(i)\mbox{ for }1\leq i< k)\nonumber\\
        &\leq & 2^kP(S_k>M(k)+f(k), \mbox{ and } \frac{1}{\sigma_1} (S_i-\frac{i}{k}S_k)\leq \frac{1}{\sigma_1}(f(i)-\frac{i}{k}f(k))\mbox{ for }1\leq i\leq k).\nonumber\\
        &&
      \end{eqnarray}
      $\frac{1}{\sigma_1}(S_i-\frac{i}{k}S_k)$ is a discrete Brownian bridge and is independent of $S_k$. Because of this independence, the above quantity is less than or equal to
      $$2^kP(S_k>M(k)+f(k))\cdot P( \frac{1}{\sigma_1}(S_i-\frac{i}{k}S_k)\leq \frac{1}{\sigma_1}(f(i)-\frac{i}{k}f(k))\mbox{ for }1\leq i< k).$$

      The first probability can be estimated similarly to \eqref{endptEst},
      \begin{eqnarray}\label{quad2est1}
      &&P(S_k>M(k)+f(k))\nonumber\\
      &\leq & \frac{C}{\sqrt{k}}\exp\left(-\frac{\left(\sqrt{2\log 2}\sigma_1 k -\frac{3\sigma_1}{2\sqrt{2\log 2}}\log k+
                         \frac{5\sigma_1}{2\sqrt{2\log 2}}\log (\frac{n}{2}-k) +y\right)^2}{2k\sigma_1^2}\right)\nonumber\\
                         &\leq &  C2^{-k}k(\frac{n}{2}-k)^{-5/2}e^{-\frac{\sqrt{2\log 2}}{\sigma_1}y}.
      \end{eqnarray}

      To estimate the second probability, we first estimate $\frac{1}{\sigma_1}(f(i)-\frac{i}{k}f(k))$. It
      is less than or equal to $\frac{1}{\sigma_1}f(i)=\frac{y}{\sigma_1}+\frac{5}{2\sqrt{2\log 2}}\log i$ for $i\leq k/2<n/4$, and, for $k/2\leq i< k$, it is less than or equal to
      \begin{eqnarray*}
      &&\frac{5}{2\sqrt{2\log 2}}\log(n/2-i)-\frac{i}{k}\frac{5}{2\sqrt{2\log 2}}\log(n/2-k)+\frac{y}{\sigma_1}(1-\frac{i}{k})\\
      &=&\frac{5}{2\sqrt{2\log 2}}\left(\log(n/2-i)-\log(n/2-k)+\frac{k-i}{k}\log(n/2-k)\right)+\frac{y}{\sigma_1} (1-\frac{i}{k})\\
      &\leq &\frac{5}{2\sqrt{2\log 2}}\left(\log(k-i)+\frac{k-i}{k}\log k\right)+\frac{y}{\sigma_1}\leq 100\log(k-i)+\frac{y}{\sigma_1}.
      \end{eqnarray*}
      Therefore, applying Lemma \ref{lem_bb}, we have
      \begin{eqnarray}\label{bbEst}
      &&P\left(\frac{1}{\sigma_1} (S_i-\frac{i}{k}S_k)\leq \frac{1}{\sigma_1}(f(i)-\frac{i}{k}f(k))\mbox{ for }1\leq i\leq  k\right)\nonumber\\
      &\leq & P\left(\frac{1}{\sigma_1} (S_i-\frac{i}{k}S_k)\leq 100\log i+\frac{y}{\sigma_1}\mbox{ for }1\leq i\leq k/2,\mbox{ and }\frac{1}{\sigma_1} (S_i-\frac{i}{k}S_k)\leq\right. \nonumber\\
      &&\mbox{   } \left.100\log(k-i)+\frac{y}{\sigma_1} \mbox{ for } k/2\leq i\leq k\right)\leq C(1+y)^2/k,
      \end{eqnarray}
      where $C$ is independent of $n$, $k$ and $y$.

      By all the above estimates \eqref{quad2Prob}, \eqref{quad2est1} and \eqref{bbEst},
      \begin{equation}\label{quad2}
      \sum_{k=n/4}^{n/2}P(\exists v\in\mathds{D}_n \mbox{ such that } S_v>M(n)+y,\tau_v=k)
      \leq C(1+y)^2e^{-\frac{\sqrt{2\log 2}}{\sigma_1}y}\sum_{k=1}^{\infty} k^{-5/2}.
      \end{equation}
      This can again be made as small as we wish, say at most $\frac{1}{100}$, by choosing $y$ large enough.

(\romannumeral3). When $n/2\leq k\leq 3n/4$, we have
         \begin{eqnarray*}
             &&P(\exists v\in\mathds{D}_n \mbox{ such that } S_v>M(n)+y,\tau_v=k)\\
             &\leq & P(\exists v\in\mathds{D}_k\mbox{ such that } S_v>M(k)+f(k)
             \mbox{ and } S_{v^i}\leq M(i)+f(i)\mbox{ for }1\leq i\leq n/2)\\
             &\leq &
              E\sum_{v\in\mathds{D}_k}1_{\{S_v>M(k)+f(k), \mbox{ and } S_{v^i}\leq M(i)+f(i)\mbox{ for }1\leq i< n/2\}}.
        \end{eqnarray*}
      The above expectation is, by conditioning on $\{S_{v^{n/2}}=M(n)+x\}$,
         \begin{eqnarray}\label{quad3Int}
          &&2^k\int_{-\infty}^yP(S_{k-n/2}'>M(k)-M(n/2)+f(k)-x)\cdot\nonumber\\
          &&\;\;\;\;\;\;\cdot P(S_i-\frac{i}{n/2}S_{n/2}\leq f(i)-\frac{i}{k}x\mbox{ for }1\leq i< n/2)\cdot\nonumber\\
          && \;\;\;\;\;\;\cdot p_{S_{n/2}}(M(n/2)+x) dx,
         \end{eqnarray}
      where $S$ and $S'$ are two copies of the random walks before and after time $n/2$, respectively, and $p_{S_{n/2}}(x)$ is the density of $S_{n/2}\sim N(0,\frac{\sigma_1^2n}{2})$.

      We then estimate the three factors of the integrand separately.
      The first one, which is similar to \eqref{endptEst}, is bounded above by
      \begin{eqnarray*}
         &&P(S_{k-n/2}'>M(k)-M(n/2)+f(k)-x)\leq \frac{C}{\sqrt{k-n/2}} e^{-\frac{\left(M(k)-M(n/2)+f(k)-x\right)^2}{2(k-n/2)\sigma_2^2}}\\
         &\leq & C 2^{-(k-n/2)}(k-\frac{n}{2})^{-3/2}e^{-\frac{\sqrt{2\log 2}}{\sigma_2}(y-x)}.
      \end{eqnarray*}
      The second one, which is similar to \eqref{bbEst}, is estimated using Lemma \ref{lem_bb},
      \begin{equation}\label{bbEstHalfway}
         P(S_i-\frac{i}{n/2}S_{n/2}\leq f(i)-\frac{i}{k}x\mbox{ for }1\leq i< n/2)\leq  C(1+2y-x)^2/n.
      \end{equation}
      The third one is simply the normal density
      \begin{equation}\label{endptEstHalfway}
      p_{S_{n/2}}(M(n/2)+x)=\frac{C}{\sqrt{n}}e^{-\frac{(M(n/2)+x)^2}{n\sigma_1^2}}\leq C 2^{-n/2}n e^{-\frac{\sqrt{2\log 2}}{\sigma_1}x}.
      \end{equation}
      Therefore, the integral term \eqref{quad3Int} is no more than
      $$C(k-n/2)^{-3/2}e^{-\frac{\sqrt{2\log 2}}{\sigma_2}y}\int_{-\infty}^{y}(1+2y-x)^2e^{(\frac{\sqrt{2\log 2}}{\sigma_2}-\frac{\sqrt{2\log 2}}{\sigma_1})x}dx,$$
      which is less than or equal to $C(1+y)^2e^{-\frac{\sqrt{2\log 2}}{\sigma_1}y}(k-n/2)^{-3/2}$ since $\sigma_2<\sigma_1$.

      Summing these
      upper bounds together, we obtain that
      \begin{equation}\label{quad3}
      \sum_{k=n/2}^{3n/4}P(\exists v\in\mathds{D}_n \mbox{ such that } S_v>M(n)+y,\tau_v=k)\leq C(1+y)^2e^{-\frac{\sqrt{2\log 2}}{\sigma_1}y}\sum_{k=1}^{\infty}k^{-3/2}.
      \end{equation}
      This can again be made as small as we wish, say at most $\frac{1}{100}$, by choosing $y$ large enough.

(\romannumeral4). When $3n/4< k\leq n$, we have
         \begin{eqnarray*}
             &&P(\exists v\in\mathds{D}_n \mbox{ such that } S_v>M(n)+y,\tau_v=k)\\
             &\leq & P(\exists v\in\mathds{D}_k\mbox{ such that } S_v>M(k)+f(k),
             \mbox{ and } S_{v^i}\leq M(i)+f(i)\mbox{ for }1\leq i<k)\\
             &\leq &
              E\sum_{v\in\mathds{D}_k}1_{\{S_v>M(k)+f(k), \mbox{ and } S_{v^i}\leq M(i)+f(i),\mbox{ for }1\leq i< k\}}.
        \end{eqnarray*}
      The above expectation is, by conditioning on $\{S_{v^{n/2}}=M(n)+x\}$,
         \begin{eqnarray*}
          &&2^k\int_{-\infty}^yP(S_{k-n/2}'>M(k)-M(n/2)+f(k)-x,\\
          &&\;\;\;\;\;\;\;\;\;\;\;\;\;\;\;\; S_i'<M(i)-M(n/2)+f(i)-x, \mbox{ for }n/2<i\leq k)\\
          &&\;\;\;\;\;\;\cdot P(S_i-\frac{i}{n/2}S_{n/2}\leq f(i)-\frac{i}{k}x\mbox{ for }1\leq i< n/2)\cdot p_{S_{n/2}}(M(n/2)+x) dx
         \end{eqnarray*}
      where $S$ and $S'$ are copies of the random walks before and after time $n/2$, respectively.

      The second and third probabilities in the integral are already estimated in \eqref{bbEstHalfway} and \eqref{endptEstHalfway}. It remains to bound the first probability. Similar to \eqref{quad2Prob}, it is bounded above by
      \begin{eqnarray*}
      &&P\left(S_{k-n/2}'>M(k)-M(n/2)+f(k)-x, S_i'<M(i)-M(n/2)+f(i)-x,\right.\\
      &&\;\;\;\;\; \mbox{ for }n/2<i\leq k\Big)\leq C(1+2y-x)^2e^{-\frac{\sqrt{2\log 2}}{\sigma_2}(2y-x)}(n-k)^{-5/2}.
      \end{eqnarray*}

      With these estimates, we obtain in this case, in the same way as in (\romannumeral3), that
      \begin{equation}\label{quad4}
      \sum_{k=3n/4}^{n}P(\exists v\in\mathds{D}_n \mbox{ such that } S_v>M(n)+y,\tau_v=k)\leq C(1+y)^2e^{-\frac{\sqrt{2\log 2}}{\sigma_1}y}\sum_{k=1}^{\infty}k^{-5/2}.
      \end{equation}
      This can again be made as small as we wish, say at most $\frac{1}{100}$, by choosing $y$ large enough.

  Summing \eqref{quad1}, \eqref{quad2}, \eqref{quad3} and \eqref{quad4}, then \eqref{decUBsum} and thus \eqref{decUB} follow. This concludes the proof of Theorem \ref{th_dec}.
\end{proof}

\section{Further Remarks}

We state several immediate generalization and open questions related to binary branching random walks in time inhomogeneous environments where the diffusivity of the particles takes more than two distinct values as a function of time and changes macroscopically.

Results involving finitely many monotone variances can be obtained similarly to the results on two variances in the previous sections. Specifically, let $k\geq 2$ (constant) be the number of inhomogeneities, $\{\sigma_i^2>0:i=1,\dots,k\}$ be the set of variances and $\{t_i>0:i=1,\dots,k\}$, satisfying $\sum_{i=1}^{k}t_i=1$, denote the portions of time when $\sigma_i^2$ governs the diffusivity.
%the set of strictly positive constants with $\sum_{i=1}^{k}t_i=1$, denoting the portion of time when $\sigma_i^2$ governs the diffusivity.
Consider binary branching random walk up to time $n$, where the increments over the time interval $[\sum_{i=1}^{j-1}t_in,\sum_{i=1}^jt_in)$ are $N(0,\sigma_j^2)$ for $1\leq j\leq k$. When $\sigma_1^2<\sigma_2^2<\dots<\sigma_k^2$ are strictly increasing, by an argument similar to that in Section \ref{sec_inc}, the maximal displacement at time $n$, which behaves asymptotically like the maximum for independent random walks with effective variance $\sum_{i=1}^kt_i\sigma_i^2$, is
$$\sqrt{2(\log 2)\sum_{i=1}^kt_i\sigma_i^2}n-\frac{1}{2}\frac{\sqrt{\sum_{i=1}^kt_i\sigma_i^2}} {\sqrt{2\log 2}}\log n+O_P(1).$$
%Without much more effort than in Section $\ref{sec_dec}$, one can derive that,
When $\sigma_1^2>\sigma_2^2>\dots>\sigma_k^2$ are strictly decreasing, by an argument similar to that in Section \ref{sec_dec}, the maximal displacement at time $n$, which behaves like the sub-maximum chosen by the previous greedy strategy (see \eqref{subMax}), is
$$\sqrt{2\log 2}(\sum_{i=1}^kt_i\sigma_i)n-\frac{3}{2}(\sum_{i=1}^k\frac{\sigma_i}{\sqrt{2\log 2}})\log n+O_P(1).$$

Results on other inhomogeneous environments are open and are subjects of further study. We only discuss some of the non rigorous intuition in the rest of this section.

In the finitely many variances case, when $\{\sigma_i^2:i=1,\dots,k\}$ are not monotone in $i$, the analysis of maximal displacement could be case-by-case and a mixture of the previous monotone cases. The leading order term is surely a result of the optimization problem \eqref{LDPshort} from the large deviation. But, the second order term may depend on the fluctuation constraints of the path leading to the maximum, as in the monotone case. One could probably find hints on the fluctuation from the optimal curve solving \eqref{LDPshort}. In some segments, the path may behave like Brownian bridge (as in the decreasing variances case), and in some segments, the path may behave like a random walk (as in the increasing variances case).

%But one can expect to follow the pattern in Section \ref{sec_inc} and \ref{sec_dec}. %First, one can solve the general optimization problem \eqref{LDPshort} explicitly, and %find the leading (velocity) term and the optimal polyline reaching the maximum. Second, %one can solve \eqref{LDPshort} up to time $\sum_{i=1}^{j}t_in$ for $j\leq k$ to find %the maximal displacement at intermediate levels. Between levels where the polyline is %strictly left behind the intermediate maxima, the walks get enough room to fluctuate %and behave like independent walks. Third, by replacing those potion of walks by %independent walks with some effective variances, the question then reduces to the %problem with decreasing variances (NOT necessarily strictly decreasing, thus requiring %a little more work).

In the case where the number of different variances increases as the time $n$ increases, analysis seems more challenging. A special case is when the variances are decreasing, for example, at time $0\leq i\leq n$ the increment of the walk is $N(0,\sigma_{i,n}^2)$ with $\sigma_{i,n}^2=2-i/n$. The heuristics (from the finitely many decreasing variances case) seem to indicate that the path leading to the maximum at time $n$ cannot be left `significantly' behind the maxima at all intermediate levels. This path is a `rightmost' path. From the intuition of \cite{FangZeitouni10}, if the allowed fluctuation is of order $n^{\alpha}$ ($\alpha<1/2$), then the correction term is of order $n^{1-2\alpha}$, instead of $\log n$ in \eqref{uni_max}. However, the allowed fluctuation from the intermediate maxima, implicitly imposed by the variances, becomes complicated as the difference between the consecutive variances decreases to zero. A good understanding of this fluctuation may be a key to finding the correction term.

\bibliography{FangBib}

\end{document}